\documentclass{paperclass}

\usepackage{lineno}
\linenumbers

\usepackage{color}
\usepackage{epsfig}
\usepackage{amsfonts,epsf,amsmath,amsthm,amsfonts,amssymb,latexsym, mathtools}

\newauthor{%
Matja\v{z} Kov\v{s}e}{%
M. Kov\v{s}e}{%
}[%
matjaz.kovse@gmail.com]

\title{Vertex Decomposition of Steiner Wiener Index and Steiner Betweenness Centrality} 

\classnbr{05C12, 05C70}

\nolinenumbers

\begin{document}


\begin{abstract}
A formula based on a vertex contributions of the Steiner $k$-Wiener index is induced by a newly introduced $k$-Steiner betweenness centrality, which measures the number of $k$-Steiner trees that include a particular vertex as a non-terminal vertex. This generalizes \v Skrekovski and Gutman's Vertex version of the Wiener Theorem and a result of Gago on the average betweenness centrality and the average distance in general graphs.
\end{abstract}

\section{Introduction}
In an undirected, finite, simple and connected graph $G$ the {\it (geodetic) distance} between two vertices $u, v \in V (G)$  is defined as the length of a shortest path connecting $u$ and $v$, and it is denoted by $d(u, v)$. For many other possibilities of defining a metric space on a graph see \cite{deza2009book}. One of the most investigated distance-based graph invariants is the {\it Wiener index}, defined as the sum of distances between all pairs of vertices of graph $G$:

\begin{equation*}
W(G)=  \sum_{u,v\in V (G)} d(u,v).
\end{equation*}

It has been introduced by Wiener \cite{wiener1947} in 1947,  when the correlations between the boiling points of paraffins and their molecular structure has been discovered and noted that in the case of a tree it can be easily calculated from the edge contributions by the following formula:

\begin{equation} \label{wienertheorem}
W(T)= \sum_{e\in E (T)} n(T_1)n(T_2),
\end{equation}

\noindent where $n(T_1)$ and $n(T_2)$ denote the number of vertices in connected components $T_1$ and $T_2$ formed by removing an edge $e$ from the tree $T$. 

Bryant and Tupper \cite{diversities2012},  have introduced a generalisation of a metric space called {\it diversity}, as a pair $(X, \delta)$ where $X$ is a set and $\delta$ is a function from the finite subsets of $X$ to $\mathbb{R}$ satisfying two axioms:
\begin{itemize}
\item[(D1)] $\delta(A) \geq 0$, and $\delta(A) = 0$ if and only if $|A|\leq 1$ and
\item[(D2)] If $B\neq 0$ then $ \delta(A \cup B) + \delta (B \cup C) \geq \delta (A \cup C)$ for all finite $A, B, C \subseteq X$. 
 \end{itemize}
Every diversity has an induced metric, given by $d(a, b) =\delta(\{a, b\})$ for all $a, b \in X$. See a forthcoming book \cite{deza2016book} for many other generalisations of finite metrics.

The classical Steiner tree problem is defined as follows: Given a set of points
in a metric space, find the shortest network interconnecting all given points.
Such a network is called a {\it Steiner tree} on the given set.
As a geometrical problem it is in a simpler form dating back already to Fermat and later to Euler and Steiner. In the last five decades it has attracted a significant research interest, due to many applications as diverse as optical and
wireless communication networks, VLSI-layout and the study of phylogenetic trees \cite{du2000book,du2008book, promel2002book, onlinecompendium}.

The Steiner diversity of a graph introduced by Chartrand et al. \cite{chartrand1989}, and called there and in the later publications Steiner distance of a graph, generalises in a natural way the concept of geodetic distance in a graph, see also \cite{goddard2011distance}. {\it The Steiner diversity} $\delta(S)$ among the vertices of $S \subseteq V (G)$ is defined as the minimum number of edges of a connected subgraph $T$ whose vertex set contains $S$. Note that this set is always a tree, called a {Steiner tree} connecting vertices of $S$. If $|S|=k$ then $T$ is also called a $k$-Steiner tree on $S$. Vertices from $S$ are called {\it terminal vertices} of Steiner tree $T$, and all other vertices of $T$ are called {\it non-terminal vertices} of $T$.
 
For a natural number $k$, $k \geq 2$, {\it the Steiner $k$-Wiener index} $SW_k(G)$ of a graph $G$ is defined as

\begin{equation*}
SW_k(G) =\sum_{\mathclap{\substack{S \subseteq V(G)\\ |S|=k}}} \delta(S).
\end{equation*}

{\it The average $k$-Steiner diversity of a graph $\overline{SW_k}(G)$} is defined as $\overline{SW_k}(G) = SW_k(G) / {n \choose k}$, where $n$ denotes the number of vertices of $G$. It has been first studied in \cite{dankelmann1996average, dankelmann1997average}.
An application of the $k$-Steiner Wiener index in chemistry  is reported in \cite{gutman2015multicenter}, where it is shown that the term $W(G)+ \lambda SW_k(G)$ provides a better  approximation for the boiling points of alkanes than $W(G)$ itself, and that the best such approximation is obtained for $k=7$. 

{\it The betweenness centrality} $B(v)$ of a vertex $v\in V (G)$ is defined as the sum of the fraction of all pairs of shortest paths that pass through $v$ across all pairs of vertices in a graph:

\begin{equation*}
B(v)= \sum_{
\substack{
x,y \in V(G)\setminus\{v\}\\
x\neq y}
}
\frac{\sigma_{x,y}(v)}{\sigma_{x,y}} ,
\end{equation*}

\noindent where $\sigma_{x,y}$ denotes the number of all shortest paths between vertices $x$ and $y$ in a graph $G$ and $\sigma_{x,y}(v)$ denotes the number of all shortest paths between vertices $x$ and $y$ in graph $G$ passing through the vertex $v$. In a case when a graph models a social or communication network, as the name suggests, it measures the centrality of a vertex in a graph, by the influence of a vertex in the dissemination of information over a network. It has been independently introduced by Anthonisse in \cite{Anthonisse1971} and by Freeman in \cite{freeman1977}, and among other applications has been applied to detect communities in networks \cite{girvan2002, newman2004}. See  \cite{gago2014} for a recent survey on results and problems on betweenness centrality and related notions. 

Vertex decomposition of the Wiener index has been studied by Caporossi et al. in \cite{caporossi2012}. \v Skrekovski and Gutman followed in \cite{skrekovski2014} with a nice formula for a vertex decomposition of Wiener index in the case of a tree and in a case of general graphs they have provided a formula which connects vertex decomposition of a Wiener index with the betweenness centrality. In terms of the average betweenness centrality and the average distance of a graph this has been first observed by Gago in \cite{gago2006} and published also in \cite{comellas2007spectral, gago2012}.

The aim of this paper is to generalise results from \cite{gago2006, skrekovski2014}  to Steiner diversity by introducing $k$-Steiner betweenness centrality. After introducing a convenient notation in the next paragraph we follow in the second and third section of the paper with formulas for edge and vertex decompositions of the Steiner $k$-Wiener index of a tree.  In fourth section a $k$-Steiner betweenness centrality is introduced and a formula for a vertex decomposition of the Steiner $k$-Wiener Index for general graphs is proved. In fifth section the average $3$-Steiner betweenness centrality of a modular graph is expressed in terms of its Wiener index. In the last section we conclude by analogous results as in the fourth section for the total Steiner betweenness centrality and the total Steiner Wiener Index.
 
For a graph $G$ let $n(G)$ denote the number of its vertices. For a forest (acyclic graph) $F$ with $p$, $p >1$, (connected) components $T_1, T_2, \ldots, T_p$ denote by $N_k(F)$ the sum over all partitions of $k$ into at least two nonzero parts of products of combinations distributed among the $p$ components of $F$:

\begin{equation*}
N_k(F)=\sum_{\substack{
l_1+l_2+\ldots+l_p=k\\
0\leq l_1,l_2,\ldots,l_p<k}} {n(T_1)\choose l_1} {n(T_2)\choose l_2} \ldots {n(T_p)\choose l_p}.
\end{equation*}

\noindent For a tree $T$ we define $N_k(T)=0$. Note that by the definition ${n\choose 0}=1$, and ${n\choose k}=0$ whenever $n<k$. For $k=2$ the notation $N_k(F)$ coincides with the one used in \cite{skrekovski2014}, but it does not necessarily coincide in the case when $k=3$.

\section{Edge decomposition of the Steiner Wiener index of a tree}

For a tree $T$ and $e\in E(T)$ let $T - e$ denote a graph obtained by removing edge $e$ from $T$. Note that $T - e$ consists of two connected components.

\begin{theorem} Let $T$ be a tree on $n$ vertices. Then,

\begin{equation}\label{SteinerWiener Theorem}
SW_k(T) = \sum_{e\in E(T)}N_k(T - e).
\end{equation}

\end{theorem}

\begin{proof}
The expression $N_k(T -  e)$ counts how many times $e$ appears as an edge in a Steiner tree connecting $k$ vertices of $T$. Since the Steiner diversity between $k$ vertices equals the number of edges of the  $k$-Steiner tree connecting them, the equality in Formula \ref{SteinerWiener Theorem} follows.
\end{proof}

Since for any edge $e$, $T - e$ consists of exactly two components $T_1$ and $T_2$ it follows that

\begin{equation*}N_k(T - e)=\displaystyle \sum_{\substack{
l_1+l_2=k\\
0< l_1,l_2<k}}{n(T_1)\choose l_1} {n(T_2)\choose l_2},\end{equation*}
and therefore in this case Formula \ref{SteinerWiener Theorem} coincides with the one from \cite{gutman2015multicenter, kovse2016SWcut, gutman2015discuss}, if moreover $k=2$ it coincides with the classical Wiener result from \cite{wiener1947}.

\section{Vertex decomposition of the Steiner Wiener index of a tree}

For a tree $T$ and $v\in V(T)$ let $T - v$ denote a graph obtained by removing $v$ from $T$. Note that $T - v$ may consists of several components and that their number equals the degree of $v$.

\begin{theorem} Let $T$ be a tree on $n$ vertices. Then,

\begin{equation}\label{vertexSteinerWiener} 
SW_k(T) = \sum_{v\in V(T)}N_k(T -  v) + (k-1)\binom{n}{k}.
\end{equation}

\end{theorem}

\begin{proof}
The expression $N_k(T -  v)$ counts how many times is $v$ a non-terminal vertex of a Steiner tree connecting $k$ vertices of $T$. Since the Steiner diversity between $k$ vertices is by $k-1$ greater than the number of non-terminal vertices in the corresponding $k$-Steiner tree, adding $k-1$ for each set of $k$ vertices, we get the sum of $k$-Steiner diversities between all $k$ sets of vertices, and the equality in Formula \ref{vertexSteinerWiener} follows
\end{proof}

Many summands in $N_k(T - v)$ might be equal to zero, since for a pendent vertex $v$ of a tree $T$ we have that $T - v$ consist of only one component and therefore by the definition $N_k(T - v) = 0$ and the same holds also when $n(T_i) < l_i$ for some $i$, $1\leq i \leq p$.

\section{Vertex decomposition of the Steiner Wiener index for general graphs}

{\it The $k$-Steiner betweenness centrality} $B_k(v)$ of a vertex $v\in V (G)$ is defined as the sum of the fraction of all $k$-Steiner trees that include $v$ as its non-terminal vertex across all combinations of $k$ vertices of $G$:

\begin{equation*}
B_k(v)= \sum_{
\substack{
A\subseteq V(G)\setminus\{v\}\\
|A|=k}
}
\frac{\sigma_A(v)}{\sigma_A},
\end{equation*}

\noindent where $\sigma_{A}$ denotes the number of all Steiner trees between vertices of $A$ in a graph $G$ and $\sigma_{A}(v)$ denotes the number of all Steiner trees between vertices of $A$ in a graph $G$ that include also the vertex $v$ as a non-terminal vertex. The term "Steiner" is included in the above definition to separate our notion from the previously studied $k$-betweenness centrality from \cite{jiang2009k-betweenness}, where $k$ bounds the length of shortest paths between a pair of vertices.

\begin{theorem} Let $G$ be a connected graph on $n$ vertices. Then,

\begin{equation}\label{maintheorem}
SW_k(G) = \sum_{v\in V(G)}B_k(v) + (k-1)\binom{n}{k}.
\end{equation}

\end{theorem}

\begin{proof}
 
For a Steiner tree $T$ on $n$ vertices and $m$ edges it holds that $m=n-1$. Let $n_t$ and $n_i$ denote the number of terminal and non-terminal vertices respectively. If follows that
$m = n_t + n_i  - 1$. Therefore for $A \subseteq V(G), |A|=k$, and for any Steiner tree $T$, with $n_i$ non-terminal vertices, whose set of terminal vertices equals $A$, it follows that $n_i  = \delta(A) - k +1$.  

The sum $\sum_{v\in V(G) \setminus A} \sigma_A(v)$ counts all possible Steiner $k$-trees between elements of $A$, which equals $\sigma_A$, with multiplicity of the number of their internal vertices, which equals $\delta(A)-k+1$. Hence it follows that
$\sum_{v\in V(G) \setminus A} \sigma_A(v)= (\delta(A) - k +1)\sigma_A$. Moreover

\begin{align*}
\sum_{v\in V(G)} B_k(v) &= \sum_{v\in V(G)} \sum_{
\substack{
A\subseteq V(G)\setminus\{v\}\\
|A|=k}
}\frac{\sigma_A(v)}{\sigma_A} \\
&=\sum_{
\substack{
A\subseteq V(G)\\
|A|=k}
} \sum_{v\in V(G) \setminus A} \frac{\sigma_A(v)}{\sigma_A} \\ 
&=\sum_{
\substack{
A\subseteq V(G)\\
|A|=k}}
\frac{1}{\sigma_A} \sum_{v\in V(G) \setminus A} \sigma_A(v)  \\
&=\sum_{
\substack{
A\subseteq V(G)\\
|A|=k}}
(\delta(A) - k +1)\\
&= SW_k(T) - (k-1)\binom{n}{k}.
\end{align*}

\end{proof}

For $k=2$ the Formula \ref{maintheorem} expresses Wiener index of graph in terms of the sum of betweenness centralities of its vertices, the main result from \cite{skrekovski2014}.

For a graph $G$ on $n$ vertices {\it the average $k$-Steiner betweenness $\overline{B_k}(G)$}  is defined as

\begin{equation*}
\overline{B_k}(G)=\frac 1n \sum_{v\in V(G)}B_k(v).
\end{equation*}

\begin{cor}
Let $G$ be a connected graph on $n$ vertices. Then,

\begin{equation}\label{averageBk}
\overline{B_k}(G) = \frac{\binom{n}{k}}{n} \left(\overline{SW_k}(G) - k+1\right).
\end{equation}

\end{cor}

For $k=2$ Equality \ref{averageBk} becomes Formula for the average betweenness centrality expressed in terms of the average distance of a graph, which has been observed in \cite{comellas2007spectral, gago2006, gago2012,skrekovski2014}.

\section{Average $3$-Steiner betweenness centrality of modular graphs}

The Wiener index of any graph $G$ can be computed in polynomial time in terms of the number of vertices, and for some graph classes, like trees and benzenoid systems, even in linear time \cite{chepoi1997benzenoid}.

The problem {\it Steiner Tree in Graphs} is described as: Given a graph $G= (V, E)$, a weight $w(e)$ (a positive integer) for each $e \in E$, a subset $A \subseteq V(G)$ and a positive integer $t$, is there a subtree of $G$ that includes all the vertices of $A$ and such that the sum of the weights of the edges in the subtree is no more than $t$? This problem belongs to classical NP-complete problems from \cite{garey1979NP}, and it remains NP-complete even if all edge weights are equal, and even for a graph class of hypercubes \cite{miler1992hypercube}. Hence for $k>2$ it is not very likely to find an efficient way to compute Steiner $k$-Wiener index and $k$-Steiner betweenness centrality for general graphs. Therefore it becomes interesting to either look for lower and upper bounds, approximations in general case and exact results in particular cases. In this section we do the later.

A graph $G$ is {\it modular} if for every three vertices $x,y,z$ there exists a vertex $w$ that lies on a shortest path between every two vertices of $x, y, z$, if moreover this vertex is unique then $G$ is called a {\it median graph} \cite{van1993theory}. It is easy to see that all modular graphs are bipartite. Examples of modular graphs are trees, hypercubes, grids, complete bipartite graphs, etc. The simplest examples of non-modular are cycles on $n$ vertices, for $n\neq 4$.

The following theorem has been observed for trees in \cite{gutman2015discuss} and generalised to median graphs in \cite{kovse2016SWcut} and to modular graphs in \cite{kovse2016SWmodular}.

\begin{theorem} 
Let $G$ be a modular graph on $n$ vertices. Then,

\begin{equation*}
SW_3(G) = \frac{n-2}{2} W(G).
\end{equation*}

\end{theorem}

Combining this with the Collorary \ref{averageBk} we get the following formula for average 3-Steiner betweenness centrality of a modular graph.

\begin{cor}
Let $G$ be a modular graph on $n$ vertices. Then,

\begin{equation}\label{averageB3}
\overline{B_3}(G) = \frac{1}{n} \left(\frac{n-2}{2} W(G) - 2\binom{n}{3} \right).
\end{equation}

\end{cor}

Hence for a modular graph one can compute in polynomial time the average 3-Steiner betweenness centrality.

For $d$-dimensional hypercube  $Q_d$ it is well known that $W(Q_d)= d4^{d-1}$ and $|V(Q_d)|=2^d$. Therefore by using Formula \ref{averageB3} we get that the average $3$-Steiner betweenness centrality of $d$-dimensional hypercube equals $\overline{B_3}(Q_d)=(2^{d-1}-1)(d2^{d-2} - \frac 13 (2^{d+1}-2))$. Since hypercubes are vertex transitive graphs it follows that for each vertex $v \in V(Q_d)$ it holds that $B_3(v)=(2^{d-1}-1)(d2^{d-2} - \frac 13 (2^{d+1}-2))$.

\section{The total Steiner betweenness centrality}

Motivated by applications of betweenness centrality in social network theory, one might be interested in all Steiner trees that include a given vertex as a non-terminal vertex. Therefore we define {\it the total Steiner betweenness centrality} (or simply Steiner betweenness centrality) $B_S(v)$ of a vertex $v\in V (G)$ as the sum of the fraction of all  Steiner trees that include $v$ as its non-terminal vertex across all combinations of vertices of $G$:

\begin{equation*}
B_S(v)= \sum_{
A\subseteq V(G)\setminus\{v\}}
\frac{\sigma_A(v)}{\sigma_A},
\end{equation*}

\noindent where $\sigma_{A}$ denotes the number of all Steiner trees between vertices of $A$ in a graph $G$ and $\sigma_{A}(v)$ denotes the number of all Steiner trees between vertices of $A$ in a graph $G$ that include also the vertex $v$ as a non-terminal vertex. {\it The average Steiner betweenness $\overline{B_k}(G)$}  is defined as $\overline{B_S}(G)=\frac 1n \sum_{v\in V(G)}B_S(v)$.

{\it Total Steiner Wiener index} (or simply Steiner Wiener index) $SW(G)$ of a graph $G$ is defined as

\begin{equation*}
SW(G) =\sum_{\mathclap{\substack{S \subseteq V(G)\\|S|\geq 2}}} \delta(S).
\end{equation*}

Set with $n$ elements has $2^n - n - 1$ subsets with cardinality at least two. Therefore we define {\it the average Steiner diversity of a graph $\overline{SW}(G)$} as $\overline{SW}(G) = \frac{1}{2^n - n - 1}SW(G)$, where $n$ denotes the number of vertices of $G$.

\begin{theorem} Let $G$ be a connected graph on $n$ vertices. Then,

\begin{equation*}\label{totalSW}
SW(G) = \sum_{v\in V(G)}B_S(v) +2^{n-1}(n-2)+1.
\end{equation*}

\end{theorem}

\begin{proof}
Using Theorem \ref{maintheorem} and the equality  $\sum_{k=2}^{n-1} (k-1)  \binom{n}{k} = (2^{n-1}-1) (n-2)$ it follows that,
\begin{align*}
SW(G) &= \sum_{k=2}^n SW_k(G)\\
&= \sum_{k=2}^{n-1} SW_k(G) + SW_n(G)\\
&=  \sum_{k=2}^{n-1} \left(\sum_{v\in V(G)}B_k(v) + (k-1)\binom{n}{k}\right) + n -1\\
&= \sum_{v\in V(G)}\sum_{k=2}^{n-1}  B_k(v) + \sum_{k=2}^{n-1}(k-1)\binom{n}{k} + n -1\\
&=  \sum_{v\in V(G)}B_S(v)+2^{n-1}(n-2)+1.
\end{align*}
 
\end{proof}

The following corollary then immediately follows.

\begin{cor}
Let $G$ be a connected graph on $n$ vertices. Then,

\begin{equation*}\label{averageBS}
\overline{B_S}(G) = \frac{1}{n} \left((2^n - n - 1)\overline{SW}(G) - 2^{n-1}(n-2)-1 \right).
\end{equation*}

\end{cor}

\section{Acknowledgment}

This work was supported by the Deutsche Forschungsgemeinschaft within the EUROCORES Programme EUROGIGA (project GReGAS) of the European Science Foundation, while the author was working at the Department of Bioinformatics, University of Leipzig, and by Max Planck Institute for Mathematics in the Sciences, Leipzig.


\begin{thebibliography}{10}
 
\bibitem{Anthonisse1971}
J. M. Anthonisse, 
The rush in a directed graph, 
Technical Report BN 9/71, Stiching Math. Centrum, Amsterdam, October 1971. 

\bibitem{diversities2012}
D. Bryant, P. F. Tupper, 
Hyperconvexity and tight-span theory for diversities, 
Advances in Mathematics 231 (2012) 3172--3198.
 
\bibitem{caporossi2012}
G. Caporossi, M. Pavia, D. Vuki\v cevi\' c, M. Segatto, 
Centrality and betweenness: Vertex and edge decomposition of the Wiener index, 
MATCH Commun. Math. Comput. Chem. 68 (2012) 293--302.

\bibitem{chartrand1989}
G. Chartrand, O.R. Oellermann, S. L. Tian, H. B. Zou, 
Steiner distance in graphs,
{\v{C}}asopis pro p{\v{e}}stov{\'a}n{\'\i} matematiky 114 (1989) 399--410.
 
\bibitem{chepoi1997benzenoid}
V. Chepoi, S. Klav\v zar,
The Wiener Index and the Szeged Index of Benzenoid Systems in Linear Time,
J. Chem. Inf. and Comp. Sci. 37 (1997), 752--755.
 
\bibitem{comellas2007spectral}
F. Comellas, S. Gago, 
Spectral bounds for the betweenness of a graph, 
Linear Algebra Appl. 423 (2007) 74--80.

\bibitem{dankelmann1996average} 
P. Dankelmann, O. R. Oellermann, H. C. Swart, 
The average Steiner distance of a graph, 
J. Graph Theory 22 (1996) 15--22.

\bibitem{dankelmann1997average} 
P. Dankelmann,  O. R. Oellermann, H. C. Swart, 
On the average Steiner distance of graphs with presribed properties, 
Discrete Appl. Math. 79 (1997) 91--103.

\bibitem{deza2009book} 
M. M. Deza, E. Deza,
Encyclopedia of distances,
Springer Berlin Heidelberg, 2009.

\bibitem{deza2016book} 
E. I. Deza, M. M. Deza, M. Dutour Sikiri\'{c},
Generalizations of Finite Metrics and Cuts,
World Scientific, 2016.

\bibitem{du2000book}
D. Du (ed.), J. M. Smith (ed.), J. H. Rubinstein (ed.), 
Advances in Steiner trees, 
Combinatorial Optimization Vol. 6., Dordrecht: Kluwer Academic Publishers, 2000.

\bibitem{du2008book}
D. Du, X. Hu,
Steiner tree problems in computer communication networks, 
Hackensack, NJ: World Scientific 2008.
   
\bibitem{freeman1977} 
L. C. Freeman, 
A set of measures of centrality based upon betweenness, 
Sociometry 40 (1977) 35--41. 

\bibitem{gago2006}
S. Gago, 
M\' etodos espectrales y nuevas medidas modelos y par\' ametros en grafos peque\~{n}o-mundo invariantes de escala, Ph.D. Thesis, Univ. Polit\`{e}cnica de Catalunya, 2006 (in Spanish).

\bibitem{gago2012}
S. Gago, J. Hurajov\' a, T. Madaras, 
Notes on the betweenness centrality of a graph, 
Math. Slovaca 62 (2012) 1--12. 

\bibitem{gago2014}
S. Gago, J. Hurajov\' a, T. Madaras, 
Betweenness Centrality in Graphs, 
in: M. Dehmer, F. Emmert-Streib (eds.), Quantitative Graph Theory: Mathematical Foundations and Applications, CRC Press, Boca Raton (2014)  233--257.

\bibitem{garey1979NP}
M.R. Garey, D.S. Johnson, 
Computers and Intractability -- A Guide to the Theory of NP-Completeness, Freeman, San Francisco, 1979. 

\bibitem{girvan2002}
M. Girvan, M. E. J. Newman, 
Community structure in social and biological networks, 
Proc. Natl. Acad. Sci. USA 99 (2002) 7821--7826.

\bibitem{goddard2011distance}
W. Goddard, O. R. Oellermann, 
Distance in graphs, in M. Dehmer (ed.), Structural analysis of complex networks, Basel: Birkh\"{a}user (2011) 49--72.

\bibitem{gutman2015multicenter}
I. Gutman, B. Furtula, X. Li, 
Multicenter Wiener indices and their applications, 
J. Serb. Chem. Soc. 80 (2015) 1009--1017. 

\bibitem{jiang2009k-betweenness}
K. Jiang, D. Ediger, D.A. Bader, 
Generalizing $k$-Betweenness Centrality Using Short Paths and a Parallel Multithreaded Implementation, in Parallel Processing, 2009, ICPP '09 - The 38th International Conference on Parallel Processing Vienna, Austria, p. 542--549.

\bibitem{kovse2016SWcut}
M. Kov\v se, 
Steiner Wiener index and a cut method, 
submited.

\bibitem{kovse2016SWmodular}
M. Kov\v se, 
Wiener index and Steiner Wiener index of a graph, 
submited.

\bibitem{gutman2015discuss}
X. Li, Y. Mao, I. Gutman, 
The Steiner Wiener index of a graph, 
Discussiones Mathematicae Graph Theory 36 (2016) 455--465.


\bibitem{miler1992hypercube}
Z. Miler, P. Manley, 
The Steiner problem in the hypercube, 
Networks 22 (1992) 1--19.

\bibitem{newman2004}
M. E. J. Newman, M. Girvan, 
Finding and evaluating community structure in networks, 
Phys. Rev. E 69 (2004) 026113.

\bibitem{skrekovski2014}
R. \v Skrekovski, I. Gutman,
Vertex Version of the Wiener Theorem,
MATCH Commun. Math. Comput. Chem. 72 (2014) 295--300.

\bibitem{promel2002book}
H. J. Pr\"{o}mel, A. Steger, 
The Steiner tree problem: a tour through graphs, algorithms, and complexity, 
Advanced Lectures in Mathematics, Braunschweig: Vieweg 2002. 
  
\bibitem{van1993theory}
M. L. J van De Vel, Theory of convex structures, North-Holland Mathematical Library 50, Amsterdam: North-Holland, 1993.
 
\bibitem{wiener1947}
H. Wiener, 
Structural determination of paraffin boiling points, 
Journal of the American Chemical Society 69 (1947) 17--20.

\bibitem{onlinecompendium}
Online compendium on Steiner Tree Problems at:
\url{http://theory.cs.uni-bonn.de/info5/steinerkompendium/}

\end{thebibliography}

\end{document}